\documentclass[reqno,final]{siamart251216}
\usepackage{amsmath, amsfonts, amssymb}

\usepackage[ruled,noend,linesnumbered,vlined,algo2e,algosection]{algorithm2e}
\usepackage{bm}
\usepackage{xspace,xfrac,mathtools}
\usepackage{booktabs} 
\usepackage{multirow}
\usepackage{wrapfig}
\usepackage{makecell}
\usepackage{tikz}
\usepackage{subcaption}
\usetikzlibrary{arrows.meta, trees}
\usetikzlibrary{fadings,shapes.arrows,shadows}
\usetikzlibrary{matrix, shapes, arrows, positioning}
\tikzset{
  hv/.style = {to path = {-|(\tikztotarget)\tikztonodes}},
  vh/.style = {to path = {|-(\tikztotarget)\tikztonodes}},
}

\tikzset{
    line/.style={draw, very thick, -latex},
    block/.style={rectangle, draw, text centered, rounded corners, fill=blue!30},
    colblock/.style={rectangle, draw, fill=green!20, text centered, rounded corners},
    col2block/.style={rectangle, draw, fill=blue!20, text centered, rounded corners},
    onslide/.code args={<#1>#2}{\only<#1>{\pgfkeysalso{#2}}},
}

\tikzset{arrowfill/.style={#1,general shadow={fill=black, shadow yshift=-0.8ex, path fading=arrowfading}}}
\tikzset{arrowstyle/.style n args={3}{draw=#2,arrowfill={#3}, single arrow,minimum height=#1, single arrow,
single arrow head extend=.3cm,}}

\Crefname{ALC@unique}{Line}{Lines}
\crefname{algocf}{alg.}{algs.}
\Crefname{algocf}{Algorithm}{Algorithms}

\DeclareMathOperator{\op}{\operatorname{op}}
\newcommand{\lldots}{,\,\ldots,\,}

\newcommand{\floor}[1]{\left\lfloor\,#1\,\right\rfloor}
\newcommand{\ceil}[1]{\left\lceil\,#1\,\right\rceil}

\newcommand{\invP}{q}
\newcommand{\N}{\mathbb{N}}

\newcommand*\abs[1]{\lvert#1\rvert}

\newcommand\F{\mathbb{F}}
\newcommand\Z{\mathbb{Z}}

\newcommand{\fflas}{\textrm{FFLAS-FFPACK}\xspace} 
\newcommand{\flint}{\textrm{FLINT}\xspace}

\newcommand{\ntl}{\textrm{NTL}\xspace}

\newcommand{\fpdbl}{\texttt{fp64}\xspace}

\DeclareMathOperator{\fl}{\bm{fl}}

\DeclareMathOperator\bsz{\bm{bitsize}}
\DeclareMathOperator\FMOD{\texttt{FMOD}}
\DeclareMathOperator{\fma}{\texttt{fma}}

\title{Multiword matrix multiplication over large finite fields in floating-point arithmetic%
\thanks{Version of the 19th of December 2025.}}

\author{
  Jérémy Berthomieu%
  \thanks{Sorbonne Université, CNRS, LIP6, F-75005, Paris, France \\
  (\email{jeremy.berthomieu@lip6.fr},
  \email{stef.graillat@lip6.fr},
  \email{dimitri.lesnoff@lip6.fr},
  \email{theo.mary@lip6.fr})}
\and
  Stef Graillat%
  \footnotemark[2]
\and
  Dimitri Lesnoff%
  \footnotemark[2]
  \thanks{Université Paris-Cité, LIPADE, 75006, Paris, France}
\and
  Theo Mary%
  \footnotemark[2]
}

\ifpdf
\hypersetup{
  pdftitle={Multiword matrix multiplication over large finite fields in floating-point arithmetic},
  pdfauthor={J. Berthomieu, S. Graillat, D. Lesnoff, and T. Mary}
}
\fi

\begin{document}

\maketitle

\begin{abstract}
  This article is concerned with the efficient computation of modular matrix multiplication $C=AB\bmod p$, a key kernel in computer algebra.
  We focus on floating-point arithmetic, which allows for using efficient matrix multiplication libraries.
  However, the existing approach is limited to primes $p$ with bitsize at most half the mantissa size (e.g., 26 bits with double precision arithmetic),
  and becomes quite inefficient when $p$ approaches this limit.
  We present a new approach that overcomes this limitation and can efficiently handle primes with larger bitsizes.
  The key idea is to use multiword decompositions $A=\sum_{i=0}^{u-1} \alpha^i A_i$ and
  $B=\sum_{j=0}^{v-1} \beta^j B_j$, which represent $A$ and $B$ as the scaled sum of
  $u$ and $v$ matrices (words) $A_i$ and $B_j$ with smaller coefficients.
  The product $C$ can then be reconstructed by computing $uv$ modular products $A_iB_j \bmod p$.
  We provide a rigorous analysis that proves the correctness of this approach for suitably chosen scaling parameters
  $\alpha$ and $\beta$. Our analysis determines the maximum bitsize of $p$
  that can be handled for a given $(u,v)$ decomposition; in particular, we show that using a $(2,2)$ decomposition
  suffices to handle bitsizes almost equal to the full mantissa size (e.g., the 26 bits limit is raised to 51 bits in double precision arithmetic).
  Moreover, we show that $(1,v)$ decompositions with $v>1$ are also of interest to handle intermediate bitsizes.
  We perform an extensive experimental analysis for various matrix shapes and prime bitsizes.
  Our performance benchmarks on both CPU and GPU architectures confirm the efficiency of the proposed approach,
  which can outperform the existing single word approach for bitsizes as low as 24, and can handle bitsizes as high as 51
  while retaining high performance.
\end{abstract}

\begin{keywords}
  matrix multiplication, multiword decomposition, modular arithmetic, finite fields, floating-point arithmetic, CPU, GPU, high-performance computing, rounding error
\end{keywords}

\begin{AMS}
  65Y05, 65Y20, 65F99, 65G50
\end{AMS}

\section{Introduction}\label{s:intro}

We are interested in efficiently computing the modular matrix product
\begin{equation}\label{eq.goal}
  C = AB \bmod p,
\end{equation}
where $p\in\N$ is prime, which is a key kernel in computer algebra problems.
Indeed, solving computer algebra problems requires efficient yet exact linear
algebra operations on rational numbers, such a
matrix inversion~\cite{BunchHopcroft74}~\cite[Chapter 16]{BurgisserClausenShokrollahi97}, linear system solving, PLUQ factorization, echelon
form, characteristic or minimal polynomial.
A direct computation with rationals is
infeasible due to the growth of intermediate coefficients~\cite[Section 5.2]{GeddesCzaporLabahn92}~\cite[Section 6.1]{MCA}. To
circumvent this issue, computations are done over a finite field of modular
integers $\Z/p\Z$, and the exact solution is reconstructed
using, for example, the Chinese remainder theorem.
Moreover, this reconstruction has a chance of not being valid for some values of $p$,
so it is desirable to handle values as large as possible to minimize this chance~\cite{Arnold03}.
Therefore, in this article, we aim to efficiently compute \eqref{eq.goal}
for large values of $p$.

To this purpose, most computer algebra systems implement elementary
arithmetic operations and linear algebra subroutines over finite fields,
see for example \flint~\cite{flint}, \ntl~\cite{ntl} and \fflas~\cite{fflas-ffpack}.
These libraries use either integer or floating-point arithmetic to represent finite field elements.
For a fixed bitsize,
floating-point arithmetic generally provides better performance due to the availability
of SIMD (Single Instruction, Multiple Data) instructions, such as SSE, AVX, and FMA,
and can take advantage of the BLAS (Basic Linear Algebra Subprograms) libraries,
which are highly optimized on modern CPUs and GPUs.
However, current floating-point approaches are limited by the restriction to
finite fields with prime moduli smaller than $2^{26.5}$, which corresponds to half the mantissa bitsize
in double-precision arithmetic.
For primes larger than $2^{26.5}$,
one can either switch to arbitrary-precision arithmetic, which is slower and
lacks the same level of hardware acceleration available to standard precision
floating-point operations, or resort to multimodular arithmetic
based on the Chinese remainder theorem (CRT)~\cite{dgls18}, which significantly
increases the number of operations.
Developing linear algebra for these primes is crucial as for some CRT schemes,
the magnitude of the integer to be computed is so large
that it exceeds the overall bitsize of all primes of bitsize less than 30.

In this article,
we propose new matrix multiplication algorithms that are able to
handle primes larger than $2^{26.5}$ while still using floating-point BLAS matrix
operations, thereby better leveraging the performance potential of multicore CPUs
and GPUs. At the same time, our proposed algorithm requires fewer operations
than multimodular, CRT-based approaches for primes having at most $51$ bits, almost the full mantissa bitsize in
double-precision arithmetic.
Thus, our algorithm outperforms existing approaches
for primes having between half and almost all the mantissa bitsize.

The key idea behind our approach is to use the matrix multiword decompositions
\begin{equation}\label{eq.intro-decomp}
  A = \sum_{i=0}^{u-1} \alpha^i A_i, \qquad
  B = \sum_{j=0}^{v-1} \beta^j B_j,
\end{equation}
for which \eqref{eq.goal} becomes
\begin{equation}\label{eq.intro-product}
  C = \sum_{i=0}^{u-1} \sum_{j=0}^{v-1} \alpha^i \beta^j A_i B_j \bmod p.
\end{equation}
With a suitable choice of the scaling parameters $\alpha$ and $\beta$, the coefficients of
matrices $A_i$ and $B_j$ can be made sufficiently small so that the products $A_iB_j \bmod p$ can be
efficiently computed with classical floating-point modular matrix multiplication algorithms.
We describe how to compute the decompositions \eqref{eq.intro-decomp} and the product \eqref{eq.intro-product}
in floating-point arithmetic, and we carry out a rigorous analysis to determine how to choose $\alpha$ and $\beta$
and to prove the correctness of the algorithms. In particular, we determine the maximum size of $p$
that can be handled depending on the number of words $u$ and $v$.
This allows for adaptively selecting $u$ and $v$ based on the size of $p$, and thus to optimize the cost of the algorithm
which is proportional to $uv$. We also present a concatenated variant of the algorithm that
stacks together the $B_j$ (respectively $A_i$) matrices to increase the
arithmetic intensity of the product, and is particularly efficient
when $B$ (respectively $A$) is
a tall-and-skinny (respectively short-and-wide) matrix. We implement the proposed algorithms on both multicore CPU and GPU architectures,
and perform numerical experiments that confirm their ability to handle primes with at most $51$ bits while retaining high performance.

The rest of this article is organized as follows.
We first describe in \Cref{s:sw}
the existing single word algorithm and its limitations. We then propose the new multiword algorithms
in \Cref{s:mw}. We report our numerical experiments in \Cref{s:bench}. Finally,
we provide some concluding remarks in \Cref{s:concl}.

\section{Existing single word algorithm and its limitations}\label{s:sw}

Throughout this article, we consider computations on integers using a floating-point arithmetic with $t$ bits of significand;
for IEEE double precision, $t=53$. We define $\F$ the set of floating-point numbers
that are nonnegative integers: this set certainly includes all integers $x$ such that $0 \leq x \le 2^t$.
We also define $\F_p = \left\{0, \ldots, p-1\right\}$ the set of nonnegative integers
less than $p$.
For the entirety of the article, we assume that $t\ge3$ and $p\ge 5$.

We denote by $\fl(\cdot)$ the result of a floating-point computation, where all operations inside parentheses are done in floating-point working precision.
We recall that floating-point operations in the IEEE 754 standard satisfy, in absence of underflow or overflow,
\begin{equation}\label{eq:stdModel}
  \fl(a \op b) = (a \op b) (1 + \eta),\quad \abs{\eta} \leq \epsilon / ( 1 + \epsilon), \quad \op \in \{+, -, \times, / \},
\end{equation}
where $\epsilon=2^{-t}$ is the unit roundoff~\cite{jeru13}.

Moreover, we
assume a fused multiply-add $\fma$ instruction
is available, where $\fma(a, b, c) \coloneqq  \fl(c+ab)$ is the correct floating-point rounding of $c+ab$.

\subsection{Modular reductions in floating-point arithmetic}

Computing exactly with finite fields elements using floating-point arithmetic
requires defining an efficient modulo operator similar to the predefined operator for integer types.
Each element
can be reduced using the $\FMOD$\footnote{\url{https://en.cppreference.com/w/cpp/numeric/math/fmod}} instruction where $\FMOD{(x,y)} = \fl\left(x - \floor{\frac{x}{y}}y\right)$.
In a finite field, we always reduce by the same modulus $y=p$
and so we may precompute its floating-point inverse $\invP = \fl(1/p)$.
This yields \Cref{alg:fpModularReductionFMA} given in \cite[Algorithm 3.1]{vdHoevenLQ2016}.

\begin{algorithm2e}[h]
  \caption{Floating-point reduction}\label{alg:fpModularReductionFMA}
  \SetSideCommentLeft{}
  \DontPrintSemicolon{}
  \SetKwInOut{Input}{Input}\SetKwInOut{Output}{Output}
  \Input{$x\in\F < 2^t$, a prime number $p\in\F < 2^{t-1}$, and $\invP = \fl(1/p)$.}
  \Output{$d\in\F_p$ such that $d = x \bmod p$.}

    $b = x\invP$\;
    $c = \floor{b}$\;
    $d = \fma(-c, p, x)$\quad\tcp{$x - cp$}\nllabel{lst:line:axpy}
    \If{$d \geq{} p$} {%
      $d=d-p$\;
    }
    \If{$d < 0$} {%
      $d=d+p$\;
    }
\end{algorithm2e}

\begin{proposition}~\label{prop:fpModularReductionFMA}
  \Cref{alg:fpModularReductionFMA} is correct for any integer input $x\in\F$ and
  a modulus $p$ such that $4 \leq p < 2^{t-1}$ and $x \leq 2^{t-2}p$.
\end{proposition}

\begin{proof}
  As $\invP=\fl(1/p)$ it follows from \eqref{eq:stdModel} that
  $\invP = (1/p) (1 + \eta_1)$ with $\abs{\eta_1} \leq \epsilon / (1 + \epsilon)$.
  Similarly as $b = \fl(x\invP)$, we have that $b = x\invP (1 + \eta_2) = (x/p)(1 + \eta_1)(1 + \eta_2)$
  with $\abs{\eta_2} \leq \epsilon / ( 1 + \epsilon)$. Approximation terms can
  be merged into one since $b = (x/p)(1 + \eta)$ with
  $\eta = \eta_1 + \eta_2 + \eta_1\eta_2$ and
  $\abs{\eta} \leq 2 \epsilon \frac{(1 + \frac{3}{2}\epsilon)}{(1+\epsilon)^2} \leq 2 \epsilon$ (as $\epsilon \geq 0$).
  As by hypothesis, $x \le 2^{t-2} p$, we have:
  $b \leq \frac{x}{p} (1 + \eta) < (1 + \frac{2\epsilon(1 + \frac{3}{2}\epsilon)}{(1+\epsilon)^2}) 2^{t-2} \leq 2^t$,
  as the expression inside the parenthesis is bounded by 3. As $b < 2^t$,
  its integer part can be stored as a floating-point number. As
  a consequence, $c$ is exactly equal to $\floor{b}$.
  By definition of the Euclidean division of $x$ by $p$, there exist some
  integers $q$ and $r$ such that $x = qp + r$ with $0 \leq r < p$.
  It follows that $b = q(1 + \eta) + \frac{r}{p}(1+\eta)$ which can be written as:

  \begin{equation}
    b = q + \underbrace{q\eta + \frac{r}{p}(1 + \eta)}_{\gamma}.
  \end{equation}

  We can deduce that $\gamma \leq 2\epsilon q + (r/p)(1 + 2\epsilon)$.
  As $r < p$ and $q \leq x/p \le 2^{t-2}$ then $2\epsilon q\leq 1/2$
  and so $\gamma \leq 3/2 + 2\epsilon < 2$ as long as $t \geq 3$. Moreover
  $\gamma \geq -2\epsilon q+ (r/p) (1 - 2\epsilon)$ and so similarly
  $\gamma \geq -1/2 - 2 \epsilon > -1$.
  We can conclude that $b$ belongs to the interval $\left]q-1, q+2 \right[$
  and so $c = \floor{b} \in \left\{ q-1, q, q+1 \right\}$. Let us now verify
  that $x - cp < 2^t$ and so is exactly representable by a
  floating-point number. If $c = q$ then $x-cp = r < p < 2^t$. If
  $c = q-1$ then $x - cp = p + r \leq 2p - 1 < 2^t$. Finally
  if $c = q + 1$ then $x - cp = r - p$ so
  $\abs{x - cp} \leq p < 2^t$.
\end{proof}

\Cref{prop:fpModularReductionFMA} improves the bounds found in \cite[Proposition 2.1]{vdHoevenLQ2016}
on both the modulus $p$ and the maximum element $x$ that can be reduced.
Indeed, in \cite{vdHoevenLQ2016},
the $p$ is limited to $2^{(t-1)/2}$
and $x$ is limited to
$2^{(t-1)/2}p$ which is smaller than our bound $2^{t-2}p$ for $t \ge 3$.

\begin{algorithm2e}[h]
  \caption{Modular product reduction~\cite[Function~3.6]{vdHoevenLQ2016}.\label{alg:fpModularProduct}}
  \SetSideCommentLeft{}
  \DontPrintSemicolon{}
  \SetKwInOut{Input}{Input}\SetKwInOut{Output}{Output}
  \Input{$x\in\F$ and $y\in\F$ satisfying $xy \le \frac{2^{t-1}}{3}p$, a prime number $p\in\F \le 2^{t-1}$, and $\invP = \fl(1/p).$}
  \Output{$e\in\F$ such that $e = xy \bmod p$.}

    $h = \fl(xy)$\;
    $l = \fma(x, y, -h)$\quad\tcp{$xy-h$}
    $b = \fl(h\invP)$\;
    $c = \floor{b}$\;
    $d = \fma(-c, p, h)$\quad\tcp{$h - cp$}
    $e = \fl(d + l)$\;
    \If{$e \geq{} p$} {%
      $e=e-p$\;
    }
    \If{$e < 0$} {%
      $e=e+p$\;
    }
\end{algorithm2e}

In some cases we need to reduce the product of two integers whose
result would overflow before reduction, that is, be larger than $2^t$
and thus not necessarily
in $\F$.
These cases can be handled with
\Cref{alg:fpModularProduct}, given in \cite[Function~3.6]{vdHoevenLQ2016}.
The next result is once more an improved version of \cite[Proposition~3.7]{vdHoevenLQ2016}.

\begin{proposition}\label{prop:fpModularProduct}
  \Cref{alg:fpModularProduct} is correct for integer input $x$ and $y$ in $\F$
  such that their product satisfies $xy \le \frac{2^{t-1}}{3}p$ and for input $p \le 2^{t-1}$ for all $t \ge 3$.
\end{proposition}
\begin{proof}
  Using error-free transformation and $\fma$, it is shown in
  \cite{Nievergelt2003ScalarFM,Ogita2005AccurateSum} that $h + l = xy$ with $\abs{l} \leq \epsilon \abs{xy}$.
  As $xy \le 2^{t-1}p$, it follows that $\abs{l} \leq p / 2$.
  By definition of $h$ and $b$, we have $h = xy (1 + \eta_1)$ and $b = (h/p)(1 + \eta_2)(1 + \eta_3)$
  with $\abs{\eta_1}, \abs{\eta_2}, \abs{\eta_3} \leq \epsilon / (1 + \epsilon)$ so that $h \leq (1 + \epsilon / (1 + \epsilon))xy$
  and $b \leq (1 + \eta_2 + \eta_3 + \eta_2 \eta_3)(h/p)$.
  As a consequence, $b \leq (1 + 3\epsilon) xy / p < 2^t$ so that $b$ is representable
  with a floating-point number and finally $c = \floor{b}$.

  Let us now write down the Euclidean division of $xy$ by $p$. By definition there exist some integers $j$ and $r$ such that $xy = jp + r$ with $0 \leq r < p$.
  It follows that $(1 + \eta_1) xy = (1 + \eta_1)jp + (1 + \eta_1) r$ and so
  $h = (1 + \eta_1)jp + (1 + \eta_1) r$. This can be written as
  $h / p = (1 + \eta_1)j + (1 + \eta_1) r / p$.

  We then have that $b = (1 + \beta) j + (1 + \beta) r / p$ with $\abs{\beta} = \abs{\eta_1 + \eta_2 + \eta_3 + \eta_1\eta_2 + \eta_2\eta_3 + \eta_1\eta_3 + \eta_1\eta_2\eta_3}\leq 3\epsilon \frac{(1 + 3\epsilon^2 + (7/3)\epsilon^3)}{(1+\epsilon)^3} \leq 3 \epsilon$ which can be written as
  \[
    b = j + \underbrace{j \beta + (r/p)(1 + \beta)}_{\alpha}.
  \]
  We can deduce that $\alpha \leq 3\epsilon j + (r/p) (1 + 3\epsilon)$. As $r < p$
  and $j \leq xy/p \leq \frac{2^{t-1}}{3}$
  then $3\epsilon j \leq 1 / 2$ and so $\alpha \leq 3/2 + 3\epsilon < 2$
  since $t \geq 3$ by assumption. Moreover $\alpha \geq - 3\epsilon j + (r/p)(1 - 3\epsilon)$
  and so similarly $\alpha \geq -1/2 - 3\epsilon > -1$.

  We can conclude that $b$ belongs to the interval $\left(j-1, j+2 \right)$
  and so\\
  $c = \floor{b} \in \left\{ j-1, j, j+1 \right\}$.
\end{proof}

Notice that the bound does not cover all products over a finite prime field with $p=2^{t-1}$. To circumvent this,
one can omit the precomputation of the inverse of the prime that introduces an additionnal floating-point error,
at the expense of some additional cycles.
Notice that for these primes, we know
that an element in the range $0, \lldots, 2^t$ may not be much larger than $p$.
For remaindering, it is faster to perform few subtractions by $p$
than computing an approximation of the quotient.

\subsection{Block matrix product}
Once we have defined a modulo operator using floating-point arithmetic, modular matrix multiplication
can be naively implemented by simply performing a reduction after each floating-point operation to ensure the size
of the integers remain bounded: given $A\in\F^{m\times k}$ and $B\in\F^{k\times n}$, $C=AB\in\F^{m\times n}$ can be computed as
\begin{equation}
  C \gets C + (a_j b_j^T \bmod p) \bmod p, \quad j=1\colon k,
\end{equation}
where $a_j$ is the $j$th column of $A$ and $b_j^T$ is the $j$th row of $B$.

This approach is however extremely inefficient since it requires as many
reductions as floating-point operations. The number of reductions can be reduced by computing instead
\begin{equation}\label{eq.blockProduct}
  C \gets C + (A_j B_j \bmod p) \bmod p, \quad j=1\colon \ceil{k/\lambda},
\end{equation}
where $A_j\in\F^{m\times \lambda}$ and $B_j\in\F^{\lambda\times n}$ are block-columns of $A$ and block-rows of $B$, respectively,
and where $\lambda$ is a block size that controls how often the reductions are performed.
When choosing the value of $\lambda$ we must ensure that the intermediate
computations do not reach the range at which integers are approximated when
written as a floating-point (numbers $x$ with exponent $e$ strictly greater
than $t$ such that $x \not\equiv 0 \bmod{2^{e-t}}$). Assuming that the
coefficients of $A$ and $B$ are
in $\F_p$ (that is, they are already reduced modulo $p$), then the coefficients of $A_jB_j$ are bounded by $\lambda(p-1)^2$ and so
it suffices to take $\lambda = \floor{2^t/(p-1)^2}$~\cite{DumasGiorgiPernet2008}.

To perform the inner reduction in \eqref{eq.blockProduct},
the result of $A_jB_j$ must be stored in a temporary workspace.
To avoid this additional workspace, one can remove this inner reduction
provided that the coefficients of $C+A_jB_j$ remain representable at all steps $j$ of the computation.
Then \eqref{eq.blockProduct} becomes
\begin{equation}
  C \gets C + A_j B_j \bmod p, \quad j=1\colon \ceil{k/\lambda},
\end{equation}
\Cref{alg:blockProduct} implements this latter approach.

\begin{algorithm2e}[h]
  \caption{Block matrix product over $\F_p$}\label{alg:blockProduct}

  \DontPrintSemicolon{}
  \SetKwInOut{Input}{Input}\SetKwInOut{Output}{Output}

  \Input{$A\in \F^{m\times k},\, B\in \F^{k\times n}, C\in \F_p^{m\times n}$,
  and a block size $\lambda$ satisfying \cref{prop:blockProduct}.}
  \Output{$C = C + AB \bmod p \in \F_p^{m\times n}$.}

    \For{$j = 1$ \KwTo{} $\ceil{k/\lambda}$}{%
        $C = C + A_j B_j$ \quad \tcp{$A_j$, $B_j$ submatrices of size $m\times\lambda$ and $\lambda\times n$}
        $C = C \bmod{p}$ \quad\ \tcp{Using \Cref{alg:fpModularReductionFMA}}
    }
\end{algorithm2e}

Computationally, \Cref{alg:blockProduct} is attractive because it mainly relies
on the efficient matrix products $A_jB_j$. Indeed, it performs $2mkn$ floating-point operations (flops) for the matrix products
and only $mn\ceil{k/\lambda}$ reductions, whose cost is thus negligible for a
sufficiently large block size $\lambda$. It is therefore crucial to determine the largest possible $\lambda$
such that the algorithm remains correct.

\begin{proposition}\label{prop:blockProduct}
  \Cref{alg:blockProduct} is correct for input matrices $A$, $B$, a prime number $p<2^{t-1}$, and a block size $\lambda$ such that
  \begin{equation}\label{eq.blockProductCondition}
  \lambda\, \max(A) \max(B) + p-1 \le 2^t,
  \end{equation}
  where the operator $\max(\cdot)$ returns the maximum coefficient of a matrix.
\end{proposition}

\begin{proof}
  At each iteration of the \verb|for| loop, each coefficient of $A_jB_j$ is
  computed as the dot product of vectors of size at most
  $\lambda$ and is thus bounded by
  $\lambda\max(A)\max(B)$. Then, it is added to a coefficient of $C$,
  which is
  bounded by $p-1$ since $C$ is reduced modulo $p$ at each iteration.
  The result is thus exact as long as the coefficients of $C+A_j B_j$ and $p$
  match the conditions of \Cref{alg:fpModularReductionFMA} on $x$ and $p$, that is,
 as long as \eqref{eq.blockProductCondition} holds and $p < 2^{t-1}$.
\end{proof}

\Cref{alg:blockProduct} is classically
used with $C = 0 \in \F_p^{m\times n}$
and with $A$ and $B$ with coefficients in $\F_p$~\cite{DumasGiorgiPernet2008}. In this case, since
$\max(A)$ and $\max(B)$ are both bounded by $p-1$,
\eqref{eq.blockProductCondition} rewrites as
$\lambda {(p-1)}^2 + p - 1 \le 2^t$, which holds for
\begin{equation}\label{eq.lambda-sw}
  \lambda = \floor{\frac{2^t - p + 1}{{(p - 1)}^2}}.
\end{equation}

This provides a sufficient condition on the maximum size of $p$.

\begin{corollary}
  Calling \Cref{alg:blockProduct} on $A\in\F_p^{m\times k}$,
  $B\in\F_p^{k\times n}$, $C=0\in\F_p^{m\times n}$ and block-size
  $\lambda$ satisfying \eqref{eq.lambda-sw} correctly returns
  $A B\in\F_p^{m\times n}$ if
  \begin{equation}\label{eq.condition-sw}
    p \lesssim 2^{t/2}.
  \end{equation}
\end{corollary}

\begin{proof}
  The result is correct if $\lambda \ge 1$, that is, if  $(p-1)^2 + p - 1 \le 2^t$.
  Since $(p-1)^2+p-1 = p(p-1)\le p^2$, every prime $p\le 2^{t/2}$ meets the
  constraint, and \eqref{eq.condition-sw} states the resulting approximation.
\end{proof}

With double precision arithmetic ($t=53$), \Cref{alg:blockProduct} can thus only
handle prime numbers less than $2^{26.5}$; actually, this upper bound can be
slightly increased to $2^{27.5}$ with a signed representation of finite fields but,
for the sake of simplicity, we restrict our study to the use of unsigned arithmetic.
For prime numbers approaching this limit,
the algorithm becomes quite inefficient since it must use a small block size $\lambda$.

In the next section we propose a new approach based on multiword arithmetic
that can handle much larger primes.

\section{New multiword algorithms}
\label{s:mw}

To overcome the limitations of the existing block matrix product algorithm,
we propose instead to rely on multiword arithmetic, which
consists in splitting the numbers into smaller parts, called \emph{words}, which
can be stored with a smaller precision (with fewer bits).
Multiword matrix multiplication algorithms are well studied in inexact
floating-point arithmetic, and have generated a renewed interest due to their ability
to emulate high precision arithmetic while exploiting efficient mixed precision GPU hardware~\cite{fhlm23,ooyo22,ooy24,uoi25,adfm25}.
However, to the best of our knowledge, using multiword arithmetic
for exact modular integer computations (based on floating-point arithmetic
and BLAS matrix operations) is a new idea, which we develop in the rest of this section.

\subsection{Multiword matrix decomposition}

Given $M \in\F_p^{m\times n}$ with coefficients bounded by $p$,
we seek to decompose it as the unevaluated sum of $u$ words $M_i$:
\[
M = \sum_{i=0}^{u-1} \alpha^i M_i,
\]
where to balance the coefficients of $M_i$ and make them as small as possible, we should
take $\alpha \approx p^{1/u}$. If $s$ bits are required to store the coefficients of $M$,
about $s/u$ bits should be sufficient to store those of $M_i$.
\Cref{alg:MWDecomposition} describes a method to obtain such a decomposition using
only floating-point arithmetic.

\begin{algorithm2e}[h]
  \caption{Multiword matrix decomposition}\label{alg:MWDecomposition}
  \SetSideCommentLeft{}
  \DontPrintSemicolon{}
  \SetAlgoLined{}
  \SetKwInOut{Input}{Input}\SetKwInOut{Output}{Output}
  \SetKwProg{Fn}{}{}{}
  \SetKwFunction{MWDecomposition}{\texttt{MW-Decomposition}}

  \Input{$M\in \F_p^{m\times n}$ and the number of words $u$.}
  \Output{$\alpha\in\N$ and $M_{0}, \ldots, M_{u-1} \in\F_p^{m\times n}$ such that $M = \sum_{i=0}^{u-1} \alpha^i M_i$.}
  $\alpha = \ceil{p^{1/u}}$\;
  $T = M$\;
  \For{$i = 0$ \KwTo{} $u-2$}{%
    $R = \floor{\frac{T}{\alpha}}$ \quad\qquad\tcp*{Quotient}
    $M_{i} = T - \alpha R$ \quad\tcp*{Remainder}
    $T = R$\;
  }
  $M_{u-1} = T$\;
\end{algorithm2e}

Under reasonable assumptions, this method produces words $M_i$ with coefficients no larger than $\alpha$,
even using floating-point arithmetic.
We begin by proving the following lemma.

\begin{lemma}~\label{prop.exactFpDiv}
  Let $0 \leq a < 2^t$ and $1\leq b$ two integers. Then computing the floor of the quotient of $a$ by $b$ in floating-point arithmetic
  is computing exactly the integer part of the quotient: $\floor{\fl(a/b)} = \floor{a/b}$
\end{lemma}

\begin{proof}
We first prove that the computed floor is not greater than the exact floor when the quotient produces a positive error.
\begin{equation}~\label{eq.factorPart}
  \fl\left(\frac{a}{b}\right) = \frac{a}{b} (1 + \eta) \leq \frac{a}{b} + \frac{\epsilon a}{(1+\epsilon)b} < \frac{a}{b} + \frac{2^t 2^{-t}}{b} < \frac{a}{b} + \frac{1}{b}.
\end{equation}

The fractional part of an integer quotient by $b$ may not be greater than $\frac{b-1}{b}$.
Indeed, let $r$ be the remainder of $a$ (integer) by $b$. It is at most equal to $b-1$.
Since $a= \floor{q} b + r$, the quotient is equal to:
\begin{equation}~\label{eq.fracPart}
q = \frac{a}{b} = \floor{q} + \frac{r}{b} \leq \floor{q} + \frac{b -1}{b}.
\end{equation}

We get the following upper bound by adding terms from both \cref{eq.fracPart} and \cref{eq.factorPart}:
\[
\frac{a}{b} (1 + \eta) < \floor{q} + \frac{b - 1}{b} + \frac{1}{b} = \floor{q} + 1.
\]

Similarly, if the quotient produces a negative error and $b$ does not divide $a$, we have:
\[
\frac{a}{b} (1 - \eta) > \floor{q} + \frac{1}{b} - \frac{1}{b} = \floor{q}.
\]
If $b$ divides $a$, no error is produced.

We have shown that no matter the error, the floating-point quotient is bounded by the exact quotient:
\[
  \floor{q} \leq \fl\left(\frac{a}{b}\right) < \floor{q} + 1.
\]
Hence, the floor of the floating-point quotient is equal to the exact floor of the quotient.
\end{proof}

We now prove the exactness of the decomposition.

\begin{proposition}\label{prop:decompValidity}
  Assuming $1<p<2^t$, \Cref{alg:MWDecomposition} computes exactly the decomposition
  \begin{equation}\label{eq.MWDecomp}
    M = \sum_{i=0}^{u-1} \alpha^i M_i
  \end{equation}
  where each matrix $M_i$ has nonnegative coefficients bounded by $\alpha=\ceil{p^{1/u}}$.
\end{proposition}

\begin{proof}
  In addition to $M_i$, we denote as $T_i$ and $R_i$ the values that
  $T$ and $R$ take at the end of iteration $i$ of the \verb|for| loop,
  with the notation $T_{-1} \coloneqq  M$. Our goal is to bound the
  coefficients of these matrices and check that no overflow occurs
  during any step
  of the computation. Note first that since $p>1$ we have $\alpha \ge
  2$.
  At any step, $T_i$ and $R_i$ have integer coefficients. Moreover, the coefficients of $T_{-1}$
  are all bounded by $p-1 < 2^t$. By definition, $\alpha = \ceil{p^{1/u}}$ is an integer
  so we can apply \cref{prop.exactFpDiv} to the computation of $R_i$, proving that
  it is computed exactly:
   \begin{equation}\label{eq.Ri-fpmodel}
     R_i = \floor{\fl(\frac{T_{i-1}}{\alpha})} =  \floor{\frac{T_{i-1}}{\alpha}}.
   \end{equation}

   Since $\alpha \ge 1$,
   we have $R_i < R_{i-1}$, which means that the coefficients of $R$
   decrease throughout the iterations. Thus
   for any $i$ we have
   \[
     \alpha R_i \le \alpha R_0 \le T_{-1} \le (p-1),
   \]
   which shows that the product $\alpha R_i$ does not overflow and hence is exact.

   Let us now also bound $M_i = T_{i-1} - \alpha R_i$ from above.
   By \eqref{eq.Ri-fpmodel}, we have
   $R_i \ge \frac{T_{i-1}}{\alpha} - 1$ and so,
   for $i=0 \colon u-2$,  $M_i  \le \alpha$.
   It only remains to bound the last word $M_{u-1} = T_{u-2} = R_{u-2}$
   from above.
   Reusing \eqref{eq.Ri-fpmodel} we obtain the recurrence relation
   \[
     R_i \le \frac{T_{i-1}}{\alpha} = \frac{1}{\alpha} R_{i-1}
   \]
   which yields
   \[
     R_i \le \frac{R_0}{\alpha^i}
     \le \frac{T_{-1}}{\alpha^{i+1}}
     \le \frac{p}{\alpha^{i+1}}.
   \]
   Using $\alpha = \ceil{p^{1/u}} \ge p^{1/u}$, we therefore obtain
   \[
     M_{u-1} = R_{u-2} \le p^{-(u-1)/u} p = p^{1/u} \le \alpha.
   \]

   We have therefore shown that no overflow occurs during the
   computation as long as $p<2^t$.
  To conclude it suffices to observe that $M_i = T_{i-1} - \alpha R_i$
  yields the recurrence relation
  $T_{i-1} = M_i + \alpha T_i$ for $i=0\colon u-2$. Hence
  \[
    T_{-1} = \sum_{i=0}^{u-2} \alpha^i M_i + \alpha^{u-1}T_{u-2}
  \]
  which yields the desired decomposition $M = \sum_{i=0}^{u-1} \alpha^i M_i$
  since $T_{-1} = M$ and $T_{u-2} = M_{u-1}$.
\end{proof}

\subsection{Multiword matrix multiplication}

We now explain how to use the multiword decomposition to compute the product
$C=AB \bmod p$ with a much less restrictive condition on the size of $p$ than
with the single word approach.

We consider a general setting where the decompositions of $A$ and $B$ can use possibly different numbers of words, denoted as $u$ and $v$ respectively.
We thus compute the decompositions
\[
  A = \sum_{i=0}^{u-1} \alpha^i A_i, \qquad
  B = \sum_{j=0}^{v-1} \beta^j B_j,
\]
where $\alpha = \ceil{p^{1/u}}$ and $\beta = \ceil{p^{1/v}}$, and where the coefficients of the words $A_i$ and $B_j$
are bounded by $\alpha$ and $\beta$ respectively.

The product $AB$ is then given as
\[
AB = \sum_{i=0}^{u-1} \sum_{j=0}^{v-1}  \alpha^i \beta^j A_iB_j.
\]
\Cref{alg:UVMW} computes this sum in place. For each pair $(i,j)$ one sets $\gamma=\alpha^i\beta^j \bmod p$ and $\delta=\gamma^{-1}\bmod p$, scales the accumulator by $\delta$, adds $A_iB_j$, and scales by $\gamma$. The identity is
$\gamma(\gamma^{-1}C+A_iB_j)=C+\gamma A_iB_j$.
The modular inverse is an integer in $\F_p$; one does not form the rational number $\gamma^{-1}$.

\Cref{alg:UVMW} is a bivariate Horner scheme.
It does not use a temporary workspace, and it needs a single modular inverse, that of the right radix, raised to the power $v$.
The result of $A_iB_j \bmod p$ is added directly into $C$, which is then scaled by $\beta$ or by $\alpha'=\alpha\beta^{-v}\bmod p$ according to the values of $i$ and $j$.
This inverse exists only when $p$ and $\beta$ are coprime.
\begin{algorithm2e}[h]
   \caption{Multiword matrix product}\label{alg:UVMW}
   \DontPrintSemicolon{}

\SetKwInOut{Input}{Input}\SetKwInOut{Output}{Output}
  \Input{$A\in \F_p^{m\times k},\, B\in \F_p^{k\times n},\, u,\, v,\, \lambda$.}
  \Output{$C = AB \bmod p\in \F_p^{m\times n}$.}
  $\alpha = \ceil{p^{1/u}}$ and $\beta = \ceil{p^{1/v}}$\;
  Decompose $A = \sum_{i=0}^{u-1} \alpha^i A_i$ \quad  \tcp{Using \Cref{alg:MWDecomposition}}
  Decompose $B = \sum_{j=0}^{v-1} \beta^j B_j$ \!\quad  \tcp{Using \Cref{alg:MWDecomposition}}
  $\alpha' = \alpha\beta^{-v} \bmod p$ \ \quad \tcp{Modular exponentiation and inverse}
  Initialize $C = 0$\;
  \For{$i = u-1$ \KwTo{} $0$}{%
    $C = \alpha' C \bmod p$ \ \quad \tcp{Using \Cref{alg:fpModularProduct}}
    \For{$j = v-1$ \KwTo{} $0$}{%
      $C = \beta C \bmod p$ \ \quad \tcp{Using \Cref{alg:fpModularProduct}}
      $C = C + A_i B_j  \bmod p$ \ \quad \tcp{Using \Cref{alg:blockProduct}
        with block size $\lambda$}\label{alg:UVMW:add}
    }
  }
  \Return{C}\;

\end{algorithm2e}
Before discussing the condition on the size of $p$ for this multiword product to be correct, we first describe a variant thereof
in \Cref{alg:UVMW-concat}. This variant concatenates the matrices
$B_j$ in order to compute the products $A_i B_j$, for a fixed $i$ and for all $j=0\colon v-1$,
as a single contiguous matrix product $A_i [B_0 \cdots B_{v-1}]$. This is potentially more efficient than computing each $A_iB_j$ product independently because
the concatenated product has a larger rightmost dimension ($nv$ instead of $n$) and thus a higher arithmetic intensity when $n$ is small.
Note that a variant where we concatenate the $A_i$ matrices
instead of the $B_j$ ones is also possible; in general one should try
to maximize the smallest of the two outer dimensions of the product, hence concatenating the $B_j$ matrices
when $n < m$ and the $A_i$ ones when $n > m$.

\begin{algorithm2e}[h]
  \caption{Multiword matrix product with concatenation}\label{alg:UVMW-concat}
  \DontPrintSemicolon{}

\SetKwInOut{Input}{Input}\SetKwInOut{Output}{Output}
  \Input{$A\in \F_p^{m\times k},\, B\in \F_p^{k\times n},\, u,\, v,\, \lambda$.}
  \Output{$C = AB \bmod p\in \F_p^{m\times n}$.}
  Compute $\alpha = \ceil{p^{1/u}}$ and $\beta = \ceil{p^{1/v}}$\;
  Decompose $A = \sum_{i=0}^{u-1} \alpha^i A_i$ \quad  \tcp{Using \Cref{alg:MWDecomposition}}
  Decompose $B = \sum_{j=0}^{v-1} \beta^j B_j$ \!\quad  \tcp{Using \Cref{alg:MWDecomposition}}
  $\alpha' = \alpha\beta^{-v} \bmod p$ \ \quad \tcp{Modular exponentiation and inverse}
  Initialize $C = 0$\;
  \For{$i = u-1$ \KwTo{} $0$}{%
    $[T_0 \cdots T_{v-1}] = A_i [B_0 \cdots B_{v-1}] \bmod p$ \quad\tcp{Using \Cref{alg:blockProduct}
      with block size $\lambda$}
    $C = \alpha' C \bmod p$ \ \quad \tcp{Using \Cref{alg:fpModularProduct}}
    \For{$j = v-1$ \KwTo{} $0$}{%
      $C = \beta C \bmod p$ \ \quad \tcp{Using \Cref{alg:fpModularProduct}}
      $C =  C + T_j \bmod p$ \quad \tcp{Using \Cref{alg:fpModularReductionFMA}}
    }
  }
  \Return{C}\;
\end{algorithm2e}

\begin{proposition}\label{prop:UVMW}
\Cref{alg:UVMW}, and the concatenated variant \Cref{alg:UVMW-concat} compute exactly $C=AB \bmod p$ under the conditions
$p \le 2^{t-2}$ and
\begin{equation}\label{eq.cAcB}
 \lambda\alpha\beta + p-1 \le 2^t.
\end{equation}
\end{proposition}

\begin{proof}
We need to check the exactness of all steps. By \Cref{prop:decompValidity} the multiword
decompositions obtained by \Cref{alg:MWDecomposition} are exact if $p<2^t$.
  By \Cref{prop:fpModularProduct}, the computation of $\alpha' = \alpha\beta^{-v} \bmod p$ using \Cref{alg:fpModularProduct}
is exact if $\alpha \le p$ and $\delta \coloneqq  \beta^{-v} \bmod{p} \le p$ is reduced modulo $p$ before applying \Cref{alg:fpModularProduct}.

To ensure $\delta$ is computed exactly, it is necessary to
  use \Cref{alg:fpModularReductionFMA} at each step of modular powering.
The scalings $\beta C$ and $\alpha' C$
are also exact since $\beta$, $\alpha'$, and all the coefficients of $C$ are all bounded by $p$.
Finally, the condition for the block product $C = C +  A_i B_j$ to be exact using \Cref{alg:blockProduct}
is given by \eqref{eq.blockProductCondition} in \Cref{prop:blockProduct}:
\[
  \lambda\, \max(A_i) \max(B_j) + p-1 \le 2^t,
\]
which yields \eqref{eq.cAcB} since by \Cref{prop:decompValidity}
$\max(A_i) \le \alpha$ and $\max(B_j) \le \beta$.
 The algorithm relies on two nested Horner's evaluation schemes to compute the polynomials $S_i \coloneqq  \sum_{j=0}^{v-1} \beta^j A_{i}B_{j}$ and $S \coloneqq  \sum_{i=0}^{u-1} \alpha^i S_i$. These evaluation schemes reduce the number of scaling operations when compared to a direct polynomial evaluation.

Finally, it is easy to check that \Cref{alg:UVMW-concat} is equivalent to
\Cref{alg:UVMW}, meaning that it
computes the same quantity under the same conditions.
\end{proof}

\Cref{prop:UVMW} provides in \eqref{eq.cAcB} a sufficient condition on the size of $p$ for the multiword product to be exact.
Neglecting the ceilings in the expressions of $\alpha$ and $\beta$, we obtain
\begin{equation}\label{eq.readableCondition}
  \lambda p^{1/u + 1/v} + p - 1\le 2^t.
\end{equation}
We use this more readable and almost correct condition to make a few comments.
\begin{itemize}
\item
Note first that by setting $u=v=1$, \eqref{eq.readableCondition} reduces to $\lambda p^2 + p \le 2^t$: with $\lambda=1$,
we thus recover the condition $p \lesssim 2^{t/2}$ of the single word algorithm (the case $\lambda=1$ of \eqref{eq.condition-sw}).
\item
  Consider now the case where $u=v=2$. Then \eqref{eq.readableCondition} becomes $(\lambda+1)p - 1 \le 2^t$.
  For $\lambda=1$ this allows $p < 2^{t-1}$.
  \Cref{alg:fpModularProduct} is exact for $p\le 2^{t-1}/3$.
  The division form of \Cref{prop:fpModularProduct} extends the range to $p\le 2^{t-2}$, and it is used only on $(2^{t-1}/3,\,2^{t-2}]$.
  Two words therefore cover every prime of bitsize at most $t-2$ ($51$ bits when $t=53$).
\item
  Interestingly, using $u=1$ and $v>1$ (or the converse) still provides
  a significant improvement to the single word condition: \eqref{eq.readableCondition}
  yields $p^{(v+1)/v}+p-1 \le 2^t$ or, neglecting the $p-1$ term, $p \lesssim 2^{tv/(v+1)}$.
  Thus for $v=2$, the condition is $p\lesssim 2^{2t/3}$, for $v=3$, it is
  $p\lesssim 2^{3t/4}$, and so on. As $v$ tends to a larger and larger number of words,
  the condition tends towards the ideal $p\lesssim 2^t$.
\item
  Finally, if we set $u=2$ and $v=3$, then \eqref{eq.readableCondition} becomes $\lambda p^{5/6}
  + p - 1 \le 2^t$. While this condition does not lead to any improvement compared with $u=v=2$
  when $\lambda=1$, it does allow for using a larger block size $\lambda$ while maintaining
  the condition $p\le 2^{t-2}$.
\end{itemize}

\subsection{Discussion on the cost of the algorithms}
\label{ss:cost}

Now that we have determined the maximum $p$ that a given pair $(u,v)$ can handle, it remains to
discuss the cost of the algorithm as a function of $u$ and $v$. \Cref{alg:UVMW} performs $uv$ matrix products
of dimensions $m\times k \times n$,  hence requiring $2uvmkn$ flops. This is a factor $uv$ more than the single word product.
The multiword product also requires $uvmn(\ceil{k/\lambda}+2)$ reductions,
which is also about a factor $uv$ more than the single word one. However, a key difference is that the block size $\lambda$ is not the same:
in the single word case $\lambda \approx 2^t/p^2$ whereas in the multiword case $\lambda \approx 2^t/p^{1/u+1/v}$. Therefore the multiword product
can use a potentially much bigger block size $\lambda$, which results
in a more efficient product since it reduces the relative cost of the reductions
and also increases the arithmetic intensity of the matrix products.
As for \Cref{alg:UVMW-concat}, it performs the same flops as \Cref{alg:UVMW}, but
is potentially more efficient thanks to an increased arithmetic intensity.

Based on this analysis, we can make some predictions on which approach is the best depending on the size of $p$. We will then
check these predictions in our experiments. Throughout this discussion we
assume $u \le v$, with the understanding that the converse is also possible.
We refer to the different variants as $(u,v)$-product.

The single word $(1,1)$-product is the least expensive and so is expected to be the best choice as long as it can use a sufficiently large block size, that is,
when $p\ll 2^{t/2}$. As $p$ approaches this limit, the $(1,1)$-product
will become increasingly less efficient until it is no longer correct. Around this limit
we should therefore switch to a multiword product with the smallest possible cost,
that is, $u=1$ and $v=2$; this $(1,2)$-product should be the best
until $p$ approaches its new limit $p \ll 2^{2t/3}$. At this point, we have the choice between increasing $u$ or $v$; since \hbox{$1\times 3 < 2\times 2$},
the $(1,3)$-product performs fewer flops than the $(2,2)$-product
and is therefore preferable as long as $p \ll 2^{3t/4}$. At this point, we again have the choice
between the $(1,4)$-product and the $(2,2)$-product, which perform the same number of flops. Since the limit for the $(1,4)$-product, $p\ll 2^{4t/5}$, is more restrictive
than that of the $(2,2)$-product, $p\ll 2^t$, the latter may seem preferable to the former.
However, when considering the concatenated variant of these algorithms,
the $(1,4)$-product increases the arithmetic intensity by a factor up to $4$, for unbalanced matrix dimensions, instead of $2$ for the $(2,2)$-product. Hence in situations where the concatenated $(2,2)$-product
remains memory bound, the concatenated $(1,4)$-product could outperform it as
long as $p\ll 2^{4t/5}$.
Finally, as mentioned before, the $(2,2)$-product remains correct for $p\le 2^{t-2}$.
As $p$ approaches this limit, the block size $\lambda$ tends to $1$.
We might therefore expect the $(2,3)$-product, the next least expensive variant,
to become more efficient for very large $p$, still under $p\le 2^{t-2}$.

\begin{table}[h]
  \centering
  \caption{Summary of the comparison between the different $(u,v)$-product variants.\label{tab.summary}}

  \begin{tabular}{lllllllll}
    \toprule
    $(u,v)$ & $(1,1)$ & $(1,2)$ & $(1,3)$ & $(1,4)$ & $(2,2)$ & $(2,3)$ \\
    \midrule
    Normalized flops ($=uv$) & 1 & 2 & 3 & 4 & 4 & 6 \\
    Approximate limit on $p$ & $2^{t/2}$ & $2^{2t/3}$ & $2^{3t/4}$ & $2^{4t/5}$ & $2^{t-2}$ & $2^{t-2}$ \\
    Limit on $\bsz(p)$ for $t=53$ & 26 & 35 & 39 & 42 & 51 & 51 \\
    Limit on $\bsz(p)$ for $t=24$ & 12 & 16 & 18 & 19 & 22 & 22 \\
    Maximum block size $\lambda$ & $2^t/p^2$ & $2^t/p^{3/2}$ & $2^t/p^{4/3}$ & $2^t/p^{5/4}$ & $2^t/p$ & $2^t/p^{5/6}$ \\
    \bottomrule
  \end{tabular}
\end{table}

We summarize this discussion in \Cref{tab.summary}, which compares for each $(u,v)$-product its normalized flops cost (equal to $uv$)
and its limit on $p$. To give a concrete indication of this limit we also print
the maximum bitsize of $p$ (that is,
the limit on $\log_2 p$ exclusive),
when the target floating-point arithmetic is double precision ($t=53$).
The table also records the same limits for $t=24$.

In summary, the following $(u,v)$-product algorithms are best used for the following bitsizes of $p$:
\begin{itemize}
  \item $\bsz{p} \in [1\phantom{0}, 26]$: use the $(1,1)$-product;
  \item $\bsz{p} \in [27,35]$: use the $(1,2)$-product;
  \item $\bsz{p} \in [36,39]$: use the $(1,3)$-product;
  \item $\bsz{p} \in [40,42]$: use the $(2,2)$-product or the $(1,4)$-product;
  \item $\bsz{p} \in [43,51]$: use the $(2,2)$-product;
  \item The $(2,3)$-product is not theoretically required;
 \item all of the above ranges should in practice be shifted down by a few bits due to the lower efficiency of the product when using a small block size.
\end{itemize}

We conclude this section by discussing the storage cost of our multiword approach.
The $(u,v)$-product requires $k(um + vn)$ entries for the input words and $mn$ entries
for the output. Thus, the more words are used, the more storage is needed: the approach presents
a trade-off between the bitsize of $p$ that is supported and the memory usage.
Moreover, the use of concatenation introduces an additional temporary workspace requiring $vmn$ entries.
Interestingly, in the case of a tall-and-skinny matrix $B$ ($n\ll m,k$), the $(1,v)$-product variants
require a negligible storage overhead compared with the storage of matrix $A$, which makes these variants
much less storage intensive than variants with $u\ge 2$, such as the $(2,2)$-product.

\section{Performance benchmarks}\label{s:bench}

\subsection{Experimental setting}
We have developed two implementations of the proposed algorithms.
The first one is written in FORTRAN and targets CPU architectures; the second
one is written in CUDA and targets NVIDIA GPU architectures.
The code and the benchmarks are freely accessible at \url{https://gitlab.lip6.fr/lesnoff/phdcode}.

The CPU code was compiled using the \verb`ifort` compiler (v19.1.3) and the Intel MKL (2019.5) library, which we used for all BLAS operations.
It was run on two Intel Xeon Gold 6248 CPUs with 20 cores each at 2.50GHz,
which have a double precision theoretical peak performance of about 1,600 Gflops/s.

The GPU code was compiled with \texttt{nvcc} (CUDA v12.6) and the flags: \verb`-arch=sm_80`, \verb`g++ 11.4.0` and \verb`-std=c++17`.
All the CUDA instructions are executed on the default stream.
We used cuBLAS for all BLAS operations.
The code was run on an NVIDIA A100 GPU,
which has a theoretical peak performance of about 19000 Gflops/s for double
precision arithmetic using tensor cores.

We have written CUDA kernels for the few operations that were not directly available through cuBLAS.
This includes in particular
kernels to perform the elementwise modular reductions and floor operations on a matrix.

As is common when comparing algorithms that perform different number of flops, we choose as performance metric
the ``effective'' Gflops/s rate, defined as
\begin{equation}
  \textrm{Effective Gflops/s} = \frac{2mkn}{t_\mathrm{avg}} \times 10^{-9}
\end{equation}
where $t_\mathrm{avg}$ is the execution time of the algorithm in seconds averaged over 10 runs
and where $2mkn$ corresponds to the number of flops performed by one matrix product of dimensions $m\times k\times n$.
This metric is best understood as a scaled inverse of the execution time; it can also provide some indication
of how well the hardware is utilized, although
care should be taken when comparing it to the theoretical Gflops/s peaks given above,
since even the (1,1)-product performs more than $2mkn$ flops (due to the modular reductions).

Since the values of the matrix coefficients do not affect the performance of
the algorithms, we simply generate them randomly. We consider two scenarios
which differ on both the matrix dimensions and what is included
in the execution time of the multiword algorithms.
\begin{itemize}
  \item Large square matrices (\Cref{s:square}): we first benchmark the algorithms
  in a general scenario involving large square matrices with $m=k=n=10016$,
with no particular application in mind. In this scenario, the execution time of the multiword algorithms
includes everything: the time for computing the product but also the time for computing the decomposition of both matrices.
Since the matrices are large and square, the former requires $O(n^3)$ flops whereas the latter only requires $O(n^2)$ flops, so that
the performance of the algorithms are driven by the performance of the product.
We do not test the use of concatenation (\Cref{alg:UVMW-concat}) in this scenario, since all matrix dimensions are large.
We use dimensions that are multiples of 32 because this leads to
more consistent and better performance on GPU.
\item Unbalanced matrices (\Cref{s:unbal}): in this second scenario,
  we consider a matrix product with unbalanced dimensions,
  $m=10923$, $k=32768$,
  and $n=32$; $B$ is thus a tall-and-skinny matrix. These dimensions
  of matrices are motivated by the polynomial system solving
  application where one needs to compute the minimal/characteristic
  polynomial of a square matrix of order $k$ but with only $m$ dense
  rows~\cite{Berthomieu2022,FaugereMouSparseFGLM2017}.
  The remaining $k-m$ rows are actually very sparse as they are
  rows of the identity matrix. This minimal/characteristic polynomial
  is computed using the block-Wiedemann
  algorithm~\cite{Coppersmith1994,hyun_block-krylov_2019} whose
  bottleneck consists in performing $2 k/n$ iterated products of the
  $m\times k$ matrix $A$ with a $k\times n$ matrix $B$, where $n\ll k$ is
   a block size parameter under our control; $n=32$ is a typical choice.
  Note that matrix $A$ is fixed throughout all iterations. Therefore,
  in this scenario, we do not include the time for computing the multiword decomposition of matrix $A$,
  which can be computed only once and reused for all iterations. We thus
  only measure the time for computing the decomposition of $B$
  and for computing the product. Again, because the product requires $O(mkn)$ flops whereas
  the decomposition of $B$ only requires $O(kn)$ flops, the cost of the decomposition of $B$ is negligible.
  In this scenario we will test the use of concatenation on matrix $B$ to increase its right dimension $n$,
  which is quite small.
\end{itemize}

Overall, our benchmark considers three scenarios (square matrices, and unbalanced matrices with or without concatenation),
for two architectures (CPU and GPU).
This leads to six different figures as summarized in \Cref{tab:figs}.

\begin{table}[h]
  \centering
  \caption{Summary of the benchmarks and the corresponding figures.\label{tab:figs}}

  \begin{tabular}{lll}
    \toprule
    & CPU & GPU \\
    \midrule
    Square matrices & \cref{fig.CPUsquare} & \cref{fig.GPUsquare} \\
    Unbalanced matrices (without concatenation) & \cref{fig.CPUunbal} & \cref{fig.GPUunbal} \\
    Unbalanced matrices (with concatenation) & \cref{fig.CPUunbal-concat} & \cref{fig.GPUunbal-concat} \\
    \bottomrule
  \end{tabular}
\end{table}

\subsection{Discussion of the results}

\subsubsection{Square matrices}
\label{s:square}

\begin{figure}
  \centering
  \includegraphics[width=.9\linewidth]{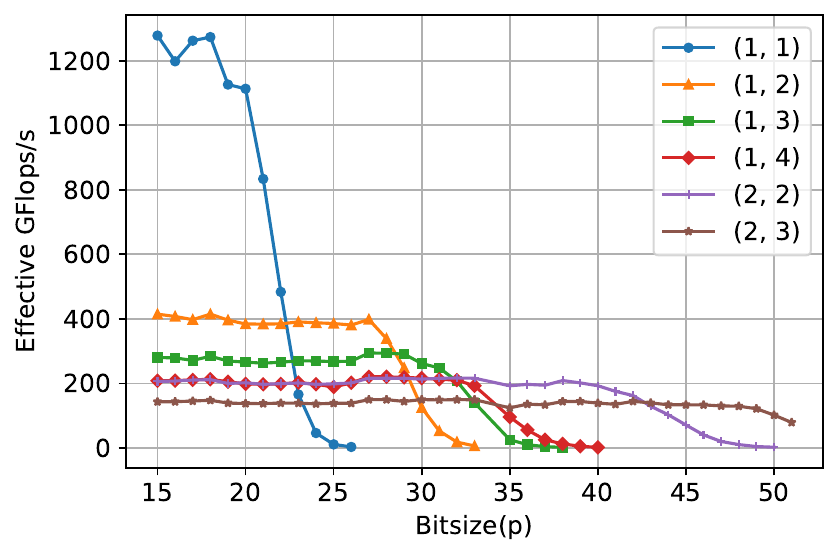}\label{fig.CPUsquare}
  \caption{Performance benchmark for square matrices on CPU.}
\end{figure}
\begin{figure}
  \centering
  \includegraphics[width=.9\linewidth]{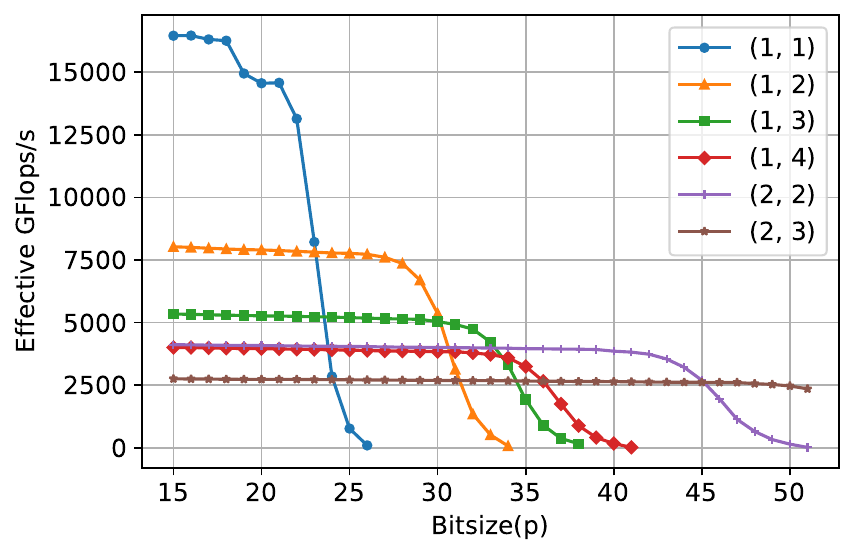}\label{fig.GPUsquare}
  \caption{Performance benchmark for square matrices on GPU.}
\end{figure}

We begin by discussing the results for square matrices on CPU (\Cref{fig.CPUsquare}).
All variants exhibit the same trend with two distinct regimes depending on the bitsize of $p$: first, a performance plateau which corresponds to the maximum
performance achievable when $p$ is small enough so that the cost of the reductions is negligible; then, a performance drop when $p$ begins approaching its limit,
due to a decreasing block size $\lambda$, which leads to a greater number of modular reductions and more inefficient matrix products.

For example, the (1,1)-product (the reference single word algorithm) achieves a performance plateau of 1200 Gflops/s
which is reasonably close to the 1600 Gflops/s theoretical peak of the hardware.
This confirms that when $p$ is small enough, the (1,1)-product is very efficient and its performance is driven by the matrix product.
However, when $p$ becomes larger, the performance drops rapidly.
Thus, although the (1,1)-product still produces correct results for primes with 24, 25 and 26 bits, the performance
in these cases is too low to be practical.

Our benchmarks therefore confirm the interest of the proposed multiword variants, which can handle larger primes while maintaining high performance.
In particular, the $(1,2)$-product outperforms the $(1,1)$-product for $\bsz(p) \ge 23$.
It achieves a performance plateau of 400 Gflops/s, about $3\times$ lower than the performance
plateau of the $(1,1)$-product.
Note that this $3\times$ time increase (which is larger than the $2\times$ flops increase) can be explained by analyzing the time breakdown
of the $(1,2)$-variant. While the $(1,1)$-variant essentially consists of a single block matrix product (\Cref{alg:blockProduct}),
the $(1,2)$-variant also requires computing the multiword decomposition of matrix $B$ and the scalings by $\delta$ and $\gamma$ with \Cref{alg:fpModularProduct}.
Despite requiring a negligible amount of flops, in practice these extra operations are less efficient than the block product and thus become non-negligible:
they represent about 26\% and 6\% of the total time for the $(1,2)$-variant, respectively.

While the $(1,2)$-product remains correct until $\bsz(p) \le 35$
the $(1,3)$-product starts outperforming it for $\bsz(p) \ge 29$, with a performance plateau of about 280 Gflops/s.
The $(1,4)$ and $(2,2)$-products both require 4 products and thus achieve the same performance plateau of about 200 Gflops/s, which starts
outperforming the $(1,3)$-product when $\bsz(p) \ge 33$.
In this scenario, the $(1,4)$-product therefore never significantly outperforms the $(2,2)$-product, which maintains its
plateau for far larger primes.
As expected, the $(2,2)$-product remains correct for all tested primes; however, its performance
eventually drops and gets surpassed by that of the $(2,3)$-product, when $\bsz(p) \ge 43$.
Even for such large primes, the $(2,3)$-product allows for
an almost constant performance of about 150 GFlops/s, which is quite satisfactory given the size of $p$.
Moreover this shows that using more than $3\times2=6$ subproducts would not be useful.

All of the above comments on the CPU benchmark also apply to the GPU one (\Cref{fig.GPUsquare}), which exhibits
similar trends. The performance of the $(1,1)$-product plateaus at 16000 Gflops/s for small primes, but is
rapidly surpassed by that of the multiword variants when $p$ gets larger.
One notable observation is that the performance plateau of the $(u,v)$-product is almost perfectly equal
to that of the $(1,1)$-product divided by $uv$, which suggests that the performance is entirely driven by the matrix product.
Thus, the $(1,2)$-product plateaus at 8000 Gflops/s, the $(1,3)$-product at 5300 Gflops/s, etc.
The points of crossover (points for which the best algorithm changes), while not exactly equal as in the CPU benchmark,
remain similar.

\subsubsection{Unbalanced matrices and effect of concatenation}
\label{s:unbal}

\begin{figure}
  \centering
  \includegraphics[width=.9\linewidth]{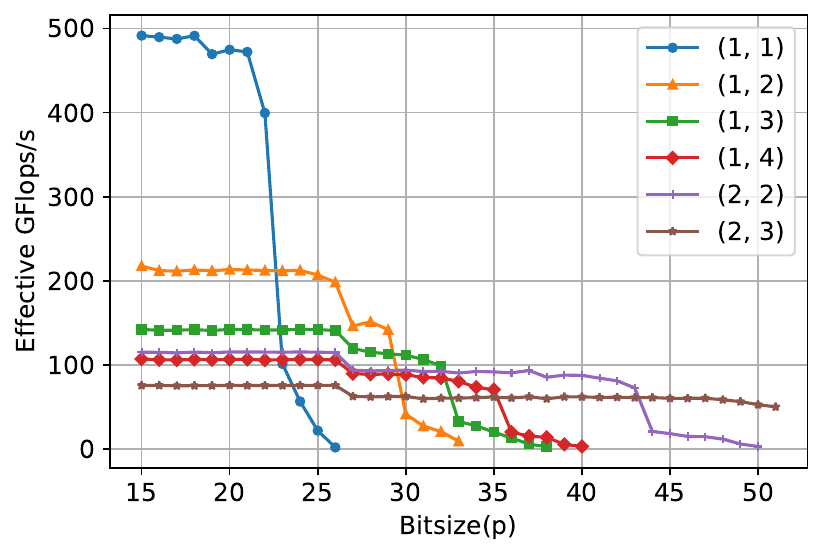}\label{fig.CPUunbal}
  \caption{Performance benchmark for unbalanced matrices on CPU.}
\end{figure}

\begin{figure}
  \centering
  \includegraphics[width=.9\linewidth]{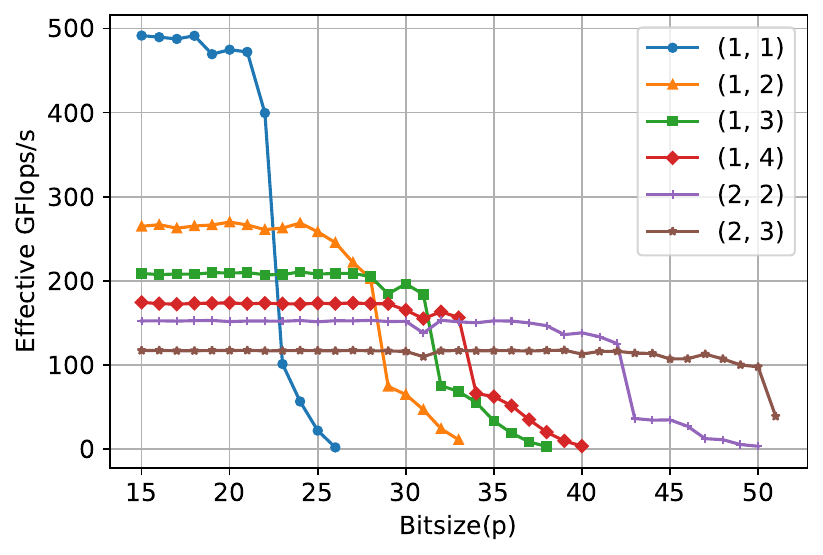}\label{fig.CPUunbal-concat}
  \caption{Performance benchmark for unbalanced matrices on CPU, with concatenation.}
\end{figure}

\begin{figure}
  \centering
  \includegraphics[width=.9\linewidth]{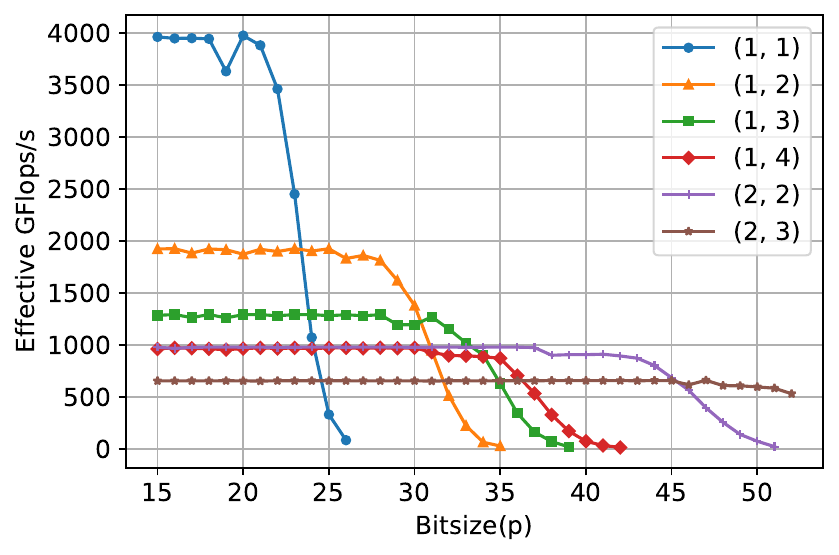}\label{fig.GPUunbal}
  \caption{Performance benchmark for unbalanced matrices on GPU.}
\end{figure}

\begin{figure}
  \includegraphics[width=.9\linewidth]{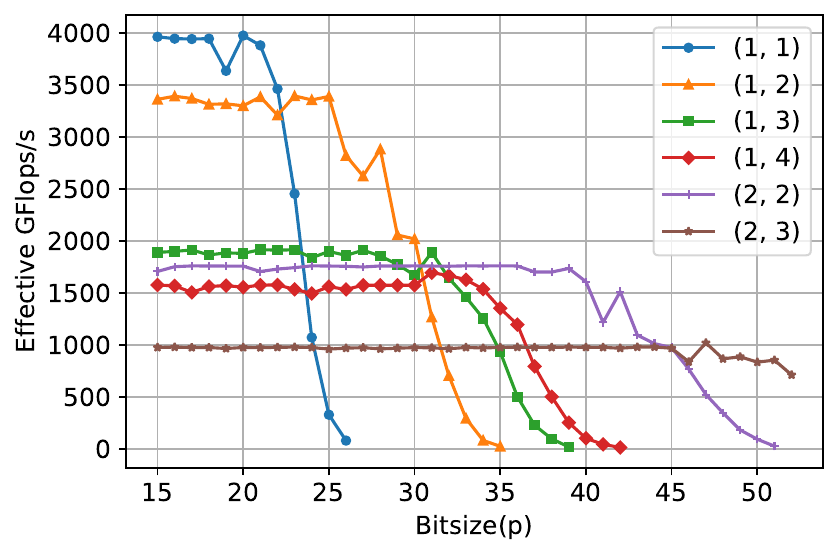}\label{fig.GPUunbal-concat}
  \caption{Performance benchmark for unbalanced matrices on GPU, with concatenation.}
\end{figure}

\Crefrange{fig.CPUunbal}{fig.GPUunbal-concat} show the performance benchmarks for unbalanced matrices.
We can observe the same overall trends as for square matrices, with one notable difference: the matrix products in this case have much lower arithmetic intensity.
Thus, the absolute performance values are smaller, that is, farther from the theoretical peak: the $(1,1)$-product
plateaus at about 500 Gflops/s on CPU (\Cref{fig.CPUunbal}) and 4000 Gflops/s on GPU (\Cref{fig.GPUunbal}).
Nevertheless, the relative performance of the multiword variants remains similar than previously and, in particular,
we confirm once more the ability of these variants to handle larger primes while retaining satisfactory performance.

Moreover, because of the lower arithmetic intensity of the product, using concatenation
in the multiword product becomes interesting.
This is illustrated in the performance benchmarks of \Cref{fig.CPUunbal-concat} (CPU) and \Cref{fig.GPUunbal-concat} (GPU),
for which we replace the multiword product (\Cref{alg:UVMW}) with its concatenated variant (\Cref{alg:UVMW-concat}).
The benchmarks show indeed that the performance of the multiword variants can be significantly improved by the use of concatenation (note that the (1,1)-product
is unaffected by this change and its performance remains identical). \Cref{tab:concatBaseline}
plots the increase of the performance plateau of the multiword variants achieved by the use of concatenation.
On CPU, we observe greater performance increases for greater values of $v$ (for example,
21\%, 47\%, and  63\% increase for the $(1, 2)$, $(1,3)$ and $(1,4)$ variants, respectively).
This is expected since a larger $v$ corresponds to a larger increase of the arithmetic intensity.
To a lesser extent, greater values of $u$ also lead to greater performance increases
(for example, 22\% vs 32\% increase for the $(1,2)$ and $(2,2)$ variants).
An interesting consequence of this behavior is that, thanks to concatenation,
the $(1,4)$-product achieves a performance plateau of 174 Gflops/s which is higher than
that of the $(2,2)$-product (152 Gflops/s). Therefore, for $\bsz(p)=32$ or $33$,
the $(1,4)$-product slightly outperforms the $(2,2)$ one (see \Cref{fig.CPUunbal-concat}).

While concatenation also leads to significant performance increases on GPU, the trend for different $(u,v)$ variants
is more unexpected. As shown in \Cref{tab:concatBaseline}, the variants with $v=2$ benefit from concatenation
much more than the other variants, especially those with $v=3$.
After investigating this surprising behavior, we have determined that this is
in fact because the cuBLAS matrix product performance is actually lower for
$n=96$ (corresponding to $v=3$) than for $n=64$ (corresponding to $v=2$).
As a result of this behavior,
the $(1,3)$ and $(1,4)$ variants are never better than the $(2,2)$ one.

\begin{table}[h]
  \centering
  \caption{Improvement of the performance plateau (Gflops/s) of multiword variants by the use of concatenation
(see \Crefrange{fig.CPUunbal}{fig.GPUunbal-concat}).}
  \label{tab:concatBaseline}
  \begin{tabular}{clccccc}
    \toprule
    & & $(1, 2)$ & $(1, 3)$ & $(1, 4)$ & $(2, 2)$ & $(2, 3)$ \\
    \midrule
    \multirow{3}{*}{CPU} & Non-concatenated (Gflops/s) & 218 & 142 & 107 & 115 & 76\\
    & Concatenated (Gflops/s) & 265 & 209 & 174 & 152 & 117 \\
    & Increase & 22\% & 47\% & 63\% & 32\% & 55\% \\
    \midrule
    \multirow{3}{*}{GPU} & Non-concatenated (Gflops/s) & 1932 & 1314 & 980 & 995 & 663 \\
    & Concatenated (Gflops/s) & 3438 & 1921 & 1600  & 1776 & 1004 \\
    & Increase & 78\% & 46\% & 63\% & 79\% & 52\% \\
    \bottomrule
  \end{tabular}
\end{table}

\subsubsection{Summary: variant selection}

\Cref{tab:synthesis} summarizes the conclusions of these experiments by indicating, for each of the six benchmarks of \Cref{tab:figs},
the range of bitsizes for which a given $(u,v)$ variant is the best.
We can see that the crossover bitsizes (where the best variant changes), while not exactly equal, are very similar from one benchmark to the other.
In particular, the existing $(1,1)$ approach is systematically outperformed before its theoretical limit of 26 bits, with crossover
bitsizes between 23 and 25. Moreover, the table also shows that each of the multiword variants considered in our benchmarks can be the best for some range
of bitsizes, which confirms the importance of adapting $(u,v)$ for optimizing the cost of the product.

\begin{table}[!h]
  \centering
  \caption{Synthesis of the bitsizes for which a given $(u,v)$ variant performs best.\label{tab:synthesis}}
  \begin{tabular}{lcccccc}
    \toprule
    & $(1,1)$ & $(1,2)$ & $(1,3)$ & $(1,4)$ & $(2,2)$ & $(2,3)$ \\
    \midrule
    Theory (\Cref{ss:cost}) & [1,26] & [27,35] & [36,39] & ---     & [40,51] & --- \\
    CPU square              & [1,22] & [23,28] & [29,31] & ---     & [32,42] & [43,51] \\
    CPU unbalanced          & [1,22] & [23,29] & [30,32] & ---     & [33,43] & [44,51] \\
    CPU unbalanced concat   & [1,22] & [23,27] & [28,31] & [32,33] & [34,42] & [43,51] \\
    GPU square              & [1,23] & [24,30] & [31,33] & ---     & [34,45] & [46,51] \\
    GPU unbalanced          & [1,23] & [24,30] & [31,33] & ---     & [34,44] & [45,51] \\
    GPU unbalanced concat   & [1,22] & [23,30] & [31,31] & ---     & [32,43] & [44,51] \\
    \bottomrule
  \end{tabular}
\end{table}

\section{Comparison with the state-of-the-art}

In this section, we discuss how our method compares
with other approaches from the state-of-the-art.

\subsection{Link with precision emulation approaches}
\label{sec.emul}

The problem of precision emulation
is to compute a floating-point matrix product $C=AB$ with high accuracy while only using low precision products.
We can distinguish several approaches depending on how this is achieved:
\begin{itemize}
\item fp32 emulation based on multiword mixed precision matrix multiply--accumulate~\cite{fhlm23,ooyo22}:
this approach decomposes $A$ and $B$ into multiple words and computes the products $A_iB_j$ with
fp32 accumulation available on NVIDIA tensor core GPUs~\cite{bhlm20}.
\item fp64 emulation based on multiword (Ozaki-I) approach~\cite{ooy24,uoi25}:
this approach also decomposes $A$ and $B$ into multiple words, but using the Ozaki scheme~\cite{ooor12}
so that the products $A_iB_j$ can be evaluated exactly; it can in particular
efficiently harness the 8-bit integers available on NVIDIA GPUs~\cite{ooy24}.
\item fp64 emulation based on multimodular (Ozaki-II) approach~\cite{oui25}: this latest approach
uses CRT-based multimodular arithmetic and is quite similar to the approach discussed in \cref{sec.CRT}.
\end{itemize}

A link can thus be made between the problem of modular matrix multiplication (the goal of this article)
and that of floating-point precision emulation. Indeed, both problems
can be tackled with either multiword or multimodular approaches.
However, there are also significant differences between the two contexts: modular matrix multiplication
involves modular reductions, which are not present in precision emulation; moreover, it must be exact,
whereas precision emulation involves several approximations. These differences make a dedicated study
of these approaches in each context necessary.

\subsection{Comparison with multimodular CRT-based approaches}
\label{sec.CRT}

\subsubsection{Number of products required by multimodular approaches}

An approach to compute the modular matrix product $C=AB\bmod p$ when $p>2^{t/2}$
is to rely on multimodular, CRT-based arithmetic~\cite{dgls18}.
The idea is to evaluate $A_i = A\bmod m_i$ and $B_i = B\bmod m_i$
for a set of $s$ coprime moduli $m_1, \ldots, m_s$.
If the moduli each satisfy
\begin{equation}\label{eq.CRT1}
\lambda (m_i-1)^2 \le 2^{t},
\end{equation}
then the products $A_iB_i$ can be computed exactly using \Cref{alg:blockProduct}
with block size $\lambda$. Then the exact
product can be recovered using the CRT
if $M = \prod_{i=1}^s m_i$ is sufficiently large.
Specifically, by the CRT we know that there exists
a unique $C$ with coefficients lower than $M$ that satisfies
$C=AB \bmod M$. Thus, if
\begin{equation}\label{eq.CRT2}
M>k(p-1)^2,
\end{equation}
then $C=AB$ is the exact product.

Putting \eqref{eq.CRT1} and \eqref{eq.CRT2} together shows that
we must have
\begin{equation}\label{eq.CRT3}
k(p-1)^2 < 2^{s(t-\log_2(\lambda))/2}
\end{equation}
and so we need at least
\begin{equation}\label{eq.numprodMM}
s = \ceil{\frac{4\log_2(p) + 2\log_2(k)}{t-\log_2(\lambda)}}
\end{equation}
moduli and thus matrix products.
Note that this lower bound may slightly underestimate the number of required products because
of the constraint that the moduli must be coprime, and hence may not all be equal
to the maximum value $2^{(t-\log_2(\lambda))/2}$.

\subsubsection{Comparison with our multiword approach}

Let us now compare
the number of products required by the multimodular approach
and by our proposed multiword approach.
Neglecting the $p-1$ term in
\cref{eq.readableCondition} shows that our approach requires
\begin{equation}\label{eq.numprodMW}
  uv = \ceil{\frac{(u+v)\log_2(p)}{t - \log_2(\lambda)}}.
\end{equation}
products.

Comparing \eqref{eq.numprodMM} and \eqref{eq.numprodMW} shows that
our approach will require fewer products than the multimodular one when
\begin{equation}
(u+v)\log_2(p) \le 4\log_2(p) + 2\log_2(k).
\end{equation}
This condition is certainly satisfied when $u+v\le 4$,
and so the (1,2), (1,3), and (2,2) multiword
variants all require fewer products than the multimodular one.
The (1,4) and (2,3) variants may also require fewer products
for large matrices for which the $2\log_2(k)$ term becomes significant.
Since the (2,2) variant can handle any prime smaller than $2^{t-\log_2(\lambda)}$,
we can conclude that our approach
is of interest for primes of bitsize between $(t-\log_2(\lambda))/2$ and $t-\log_2(\lambda)$, that is,
for primes roughly between half and the full mantissa bitsize $t$ (shifted down
by a few bits depending on the desired block size $\lambda$).

We illustrate this comparison in \Cref{fig.nbprod-comparison} for $\lambda=1$
(block size leading to the lowest possible number of products)
and $\lambda=512$ (block size that should be sufficient to attain good performance
in many settings).

Moreover, note that the multimodular approach cannot concatenate
different products together as in the multiword one, because all products
involve different matrices. Therefore, when the number of required products
is the same for both approaches, the multiword one seems preferable,
especially for matrices with unbalanced dimensions.

\begin{figure}
\begin{subfigure}{\textwidth}
  \centering
  \includegraphics[width=0.9\textwidth]{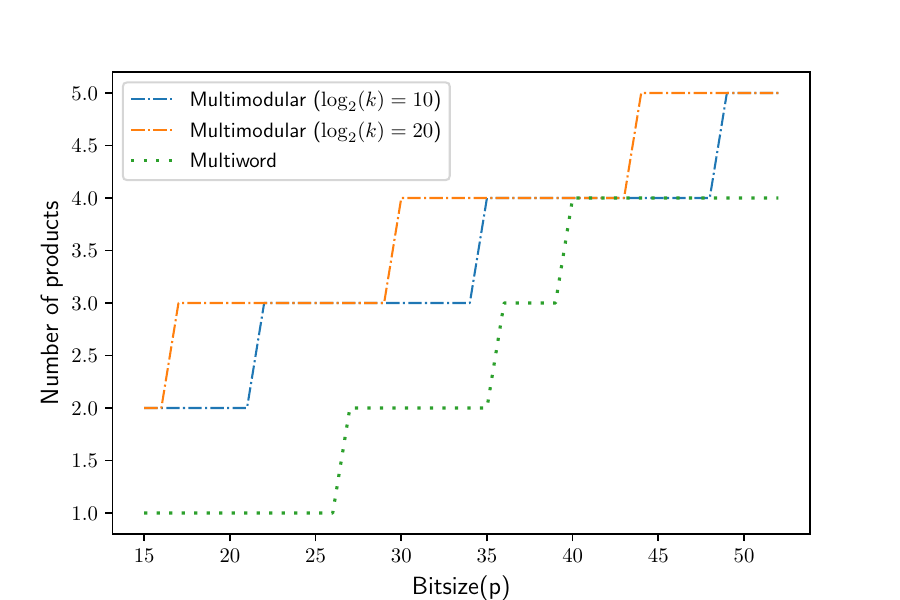}
  \caption{$\lambda=1$}
\end{subfigure}
\begin{subfigure}{\textwidth}
  \centering
  \includegraphics[width=0.9\textwidth]{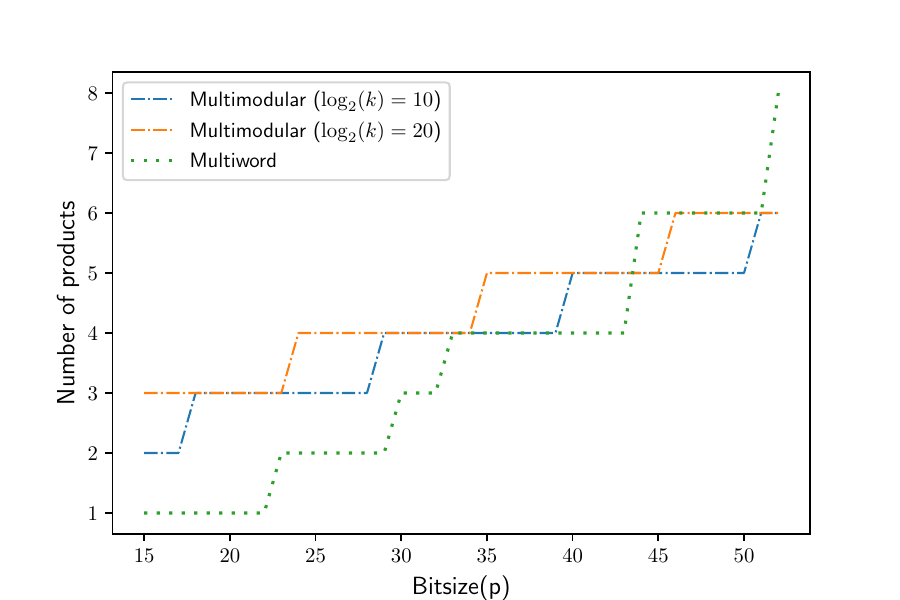}
  \caption{$\lambda=512$}
\end{subfigure}
  \caption{Number of products required by the multimodular and multiword approaches.\label{fig.nbprod-comparison}}
\end{figure}

\subsection{Experimental comparison with \flint{} and \fflas{}}

We conclude our experiments by comparing our code with the \flint{} and
\fflas{} libraries in \cref{fig:stComp}.

For \flint{}, we use the \verb|nmod_mat_mul_blas| function that does not rely on
Strassen's algorithm. For \fflas{}, we use both the unsigned \fpdbl{} (not
as fast as signed \fpdbl{} but provides the closest comparison) for bitsizes
below 26 bits and the implementation using CRT-based multiprecision.
For our approach, we take for each prime size
the $(u,v)$ variant that achieves the best performance.

We compare the three codes in sequential (1 thread only)
but with SIMD vectorization activated. Deactivating multithreading enables to
focus the comparison on the method without having results impacted
by potential multithreading issues.

This results in the benchmarks in \cref{fig:stComp}. We can see that
our code outperforms both \flint{} and \fflas{} for primes with a bitsize
between 24 and 44 bits and has a similar performance for other bitsizes, with the exception of the transition regime between single word
and multiword (22-23 bits) for which our code is not optimized.
The performance gains achieved by our approach are particularly significant for
the interval corresponding to the use of the $(1, 2)$ variant (for bitsizes
between 24 and 29 bits), leading up to a twofold improvement, from $18$ to $36$ Gflops/s.

Note that for primes up to about 20 bits, our code, \flint{} and \fflas{}
achieve a single-thread performance of about 70 Gflops/s, which is close to the
maximum theoretical CPU peak of the machine (80 Gflops/s) using AVX512.
\fflas{} achieves up to 130 Gflops/s for the very small primes,
as they automatically switch the underlying type to use fp32 arithmetic.

\begin{figure}[!ht]
  \centering
\includegraphics[width=\textwidth]{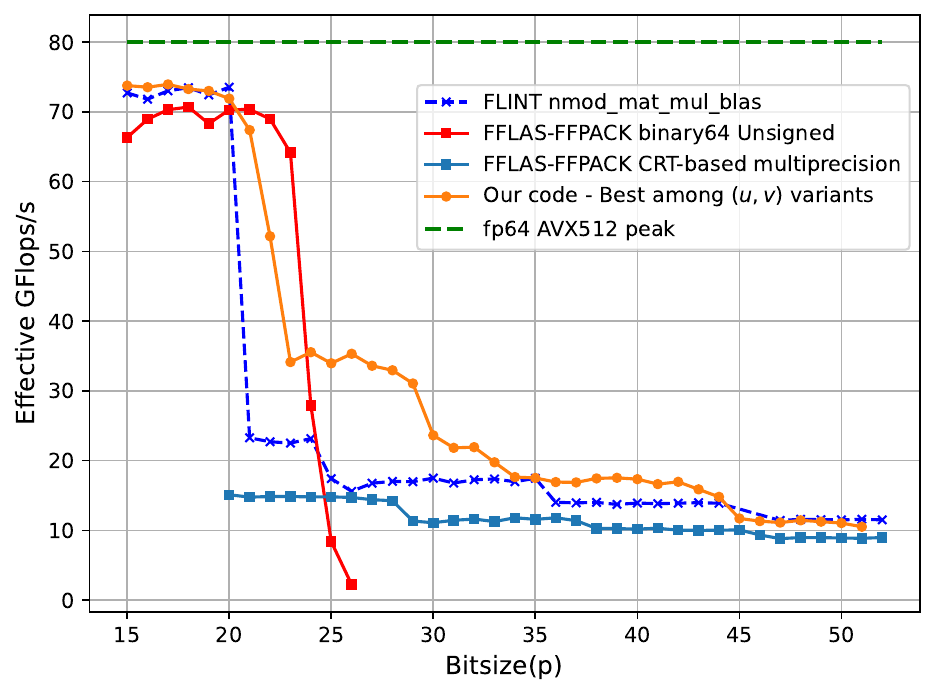}
  \caption{Single-thread comparison with FLINT and FFLAS-FFPACK libraries for square matrices (of size $10\,000$) on the same Intel Xeon Gold 6248.}
  \label{fig:stComp}
\end{figure}

Finally, it is worth recalling that unlike FLINT and FFLAS, our code supports GPUs.

In summary,
these experimental results confirms our theoretical
expectations from the previous section, showing that our approach outperforms the state of the art
based on multimodular arithmetic for medium bitsizes.

\section{Conclusion}
\label{s:concl}

We have presented a new approach to efficiently compute modular matrix multiplication $C=AB\bmod p$ in floating-point arithmetic.
The existing single word product is limited, in the ranges used above, to bitsizes of $p$ up to 26.
and becomes very inefficient as $p$ approaches this limit.
We have proposed in \Cref{alg:UVMW} a new multiword product that
decomposes $A$ and $B$ into $u$ and $v$ words, respectively,
and computes $C$ with $uv$ modular matrix products. We have also described a concatenated variant in \Cref{alg:UVMW-concat}
which can be more efficient when the products have low arithmetic intensity.
We have proved in \Cref{prop:UVMW} the correctness of this approach
and determined the maximum size of $p$ that can be handled
for a given $(u,v)$ choice.
As summarized in \Cref{tab.summary},
our multiword approach allows for handling bitsizes as large as 51,
and its cost can be optimized by adapting $(u,v)$ depending on the size of $p$.
Our performance benchmarks on CPU and GPU architectures (see \Cref{tab:figs}) confirm the efficiency
of this new approach.

This work opens several perspectives for further performance improvements.
First, the block products $A_jB_j$ in \Cref{alg:blockProduct}
could be computed in parallel via batched matrix products kernels,
at the cost of extra memory storage.
Second, the multiword approach could be extended to perform the $A_iB_j$ matrix products
in lower precision arithmetic. While this would require a greater number of words (and therefore matrix products)
to handle a given bitsize of $p$, it would also allow the use of low precision hardware, in particular GPU tensor cores~\cite{bhlm20}.

\section*{Acknowledgements}
This work was performed using HPC resources from\\ GENCI-IDRIS (Grant
AD010614986R1).
It was partially supported by the
the joint ANR-FWF
\textsc{ECARP} (ANR-19-CE48-0015) project,
and by the
EAGLES (ANR-22-CE91-0007),
\textsc{De Rerum Natura} (ANR-19-CE40-0018),
InterFLOP (ANR-20-CE46-0009),
NuSCAP (ANR-20-CE48-0014),
MixHPC (ANR-23-CE46-0005-01),
and NumPEx Exa-MA (ANR-22-EXNU-0002)
projects of the French National Agency for Research (ANR).

\bibliographystyle{siamplain}
\bibliography{strings,refs,tmary,DoctoratLip6}

\newpage{}

\end{document}